\documentclass[12pt,oneside]{amsart}
\usepackage{amssymb,amscd}

\usepackage{cite, verbatim}
\usepackage{enumerate}

\newtheorem{theorem}{Theorem}[section]
\newtheorem{lemma}[theorem]{Lemma}

\newtheorem{corollary}[theorem]{Corollary}

\theoremstyle{definition}

\theoremstyle{remark}

\newtheorem{example}[theorem]{Example}

\def\ov{\overline}

\title[The structure of a finite group and the maximum $\pi$-index of its elements]{The structure of a finite group and the maximum $\pi$-index of its elements}
\author{A-Ming Liu}
\address{School of Mathematics and Statistics, Hainan University, Haikou 570228, Hainan, P. R. China}
\email{amliu@hainanu.edu.cn}

\author{Andrey V. Vasil'ev}
\address{Sobolev Institute of Mathematics, Novosibirsk 630090, Russia}
\email{vasand@math.nsc.ru}

\begin{document}

\begin{abstract}
Given a set of primes $\pi$, the {\em$\pi$-index} of an element $x$ of a finite group $G$ is the $\pi$-part of the index of the centralizer of $x$ in~$G$. If $\pi=\{p\}$ is a singleton, we just say the {\em$p$-index}. If the $\pi$-index of $x$ is equal to $p_1^{k_1}\ldots p_s^{k_s}$, where $p_1,\ldots,p_s$ are distinct primes, then we set $\epsilon_\pi(x)=k_1+\ldots+k_s$. In this short note, we study how the number $\epsilon_\pi(G)=\max\{\epsilon_\pi(x):x\in G\}$ restricts the structure of the factor group $G/Z(G)$  of $G$ by its center. First, for a finite group $G$, we prove that  $\phi_p(G/Z(G))\leq\epsilon_p(G)$, where  $\phi_p(G/Z(G))$ is the Frattini length of a Sylow $p$-subgroup of $G/Z(G)$. Second, for a $\pi$-separable finite group $G$, we prove that $l_\pi(G/Z(G))\leq\epsilon_\pi(G)$, where $l_{\pi}(G/Z(G))$ is the $\pi$-length of $G/Z(G)$.\smallskip

\smallskip

\end{abstract}

\maketitle

\section{Introduction}\label{s:intro}

Given a finite group $G$ and $x\in G$, we refer to the size $|x^G|$ of the conjugacy class $x^G$ as the {\em index} of $x$, motivating by the well-known fact that it is equal to the index of the centralizer $C_G(x)$ of $x$ in $G$. At the very beginning of the 20th century, Burnside proved that a group of composite order, having an element whose index is a prime power, cannot be simple; it was the key ingredient in his famous theorem on the solvability of finite $\{p,q\}$-groups~\cite{Burnside}. Since then, the analysis of connection between the structure of a group and the properties of the indices of its elements has become a popular topic in finite group theory. The presented short note is within this research field. We are particulary interested in those properties of a group that follow from the restrictions on $\pi$-parts of the indices, where $\pi$ is some set of primes. All groups below are assumed to be finite.

To proceed, we need the following notation. For a natural number~$n$ and a set of primes~$\pi$, let us denote by $n_\pi$ the $\pi$-part of $n$ and by $n_{\pi'}$ the $\pi'$-part of $n$, that is, if $n=p_1^{k_1}\ldots p_s^{k_s}q_1^{m_1}\ldots q_t^{m_t}$, where $p_1,\ldots,p_s\in\pi$ and $q_1\ldots,q_t\in\pi'$, then $n_\pi=p_1^{k_1}\ldots p_s^{k_s}$ and  $n_{\pi'}=q_1^{m_1}\ldots q_t^{m_t}$. The {\em$\pi$-exponent} of~$n$ is the number $\exp_\pi(n)=k_1+\ldots+k_s$. If the set $\pi$ contains only one prime~$p$, then we say the $p$-part, $p'$-part, and {\em$p$-exponent} of~$n$.

For a group $G$, element $x\in G$, and set of primes $\pi$, we refer to the number $|x^G|_\pi$ as the \emph{$\pi$-index} of $x$ in~$G$. Put $\epsilon_\pi(x)=\exp_\pi(|x^G|)$ and $\epsilon_\pi(G)=\max\{\epsilon_\pi(x):x\in G\}$. In the case $\pi=\{p\}$, the $\pi$-index is called {\em$p$-index} of $x$ in $G$, $\epsilon_p(x)=\exp_p(|x^G|)$, and $\epsilon_p(G)=\max\{\epsilon_p(x):x\in G\}$.

As usual, $\Phi(G)$ stands for the Frattini subgroup of $G$, that is the intersection of all maximal subgroups of $G$. Put $\Phi_0(G)=G$, and for every positive integer $m$, define $\Phi_m(G)=\Phi(\Phi_{m-1}(G))$, the {\em$m$th Frattini subgroup} of $G$. Set $\phi_p(G)$ for the least integer $m$ such that $\Phi_m(P)=1$ for a Sylow $p$-subgroup $P$ of $G$.

We are ready to formulate the first main result of the paper.

\begin{theorem}\label{t:frattini}
If $G$ is a finite group, then $\phi_p(G/Z(G))\leq\epsilon_p(G)$.
\end{theorem}

It is worth noting that Theorem~\ref{t:frattini} generalizes \cite[Lemma~4]{09Vas}. See some other generalizations of that lemma in~\cite{22Go,23VG}.

Let $d_p(G)$ denote the derived length of a Sylow $p$-subgroup $P$ of a group $G$, and let $e_p(G)$ be the integer such that $p^{e_p(G)}$ is the maximum order of elements in $P$.     Since $d_p(G)\leq\phi_p(G)$ and $e_p(G)\leq\phi_p(G)$ (see Lemma~\ref{l:de} below), the following holds.

\begin{corollary}\label{c:p}
If $G$ is a finite group, then  $d_p(G/Z(G))\leq\epsilon_p(G)$ and $e_p(G/Z(G))\leq\epsilon_p(G)$.
\end{corollary}

For a set of primes $\pi$, a group $G$ is said to be {\em$\pi$-separable}, if $G$ has a normal series each of whose factors is either a $\pi$-group or $\pi'$-group. Among all such normal series there is the special one (it is shortest, in particular) called the {\em upper $\pi$-series} and defined as follows:
\begin{equation}\label{eq:r}
1=P_0\leqslant K_0<P_1< K_1<\ldots< P_l\leqslant K_l=G,
\end{equation}
where $K_i/P_i=O_{\pi'}(G/P_i)$ and $P_{i+1}/K_i=O_\pi(G/K_i)$, $i=0,\ldots l$, are the largest normal $\pi'$-subgroup of $G/P_i$ and $\pi$-subgroup of $G/K_i$, respectively. The number $l$, the least integer such that $K_l=G$, is called the {\em$\pi$-length} of $G$ and is denoted by $l_\pi(G)$. A $\pi$-separable group is called {\em$\pi$-solvable}, if the $\pi$-factors $P_{i+1}/K_i$ are solvable groups. If $\pi=\{p\}$, then we say {\em$p$-solvable} (there is no any distinction with $p$-separability in this case), the {\em upper $p$-series}, the {\em$p$-length} $l_p(G)$ (the definitions and notations from this paragraph originate from the classical Hall--Higman paper~\cite{56HalHig}).

If $G$ is a $p$-solvable group, then $l_p(G)\leq d_p(G)$; it was proved in \cite{56HalHig} for odd $p$ and in~\cite{81Br} for $p=2$. Thus, the $p$-length of a $p$-solvable group is also bounded from above by $\epsilon_p(G)$.

If $G$ is a $\pi$-solvable group, then the complete analog of the famous theorem of P. Hall holds for Hall $\pi$-subgroups of~$G$: they exist, are conjugate to each other, and every $\pi$-subgroup is contained in some Hall $\pi$-subgroup, see, e.g. \cite[Subsection 6.3]{68Gor}. Moreover, the inequality $l_\pi(G)\leq d_\pi(G)$, where $d_\pi(G)$ is the derived length of a Hall $\pi$-subgroup, was also established provided $2\not\in\pi$, see \cite{93ChP, 96Kaz}. However, this inequality cannot be converted in the inequality between the $\pi$-length of $G/Z(G)$ and $\epsilon_\pi(G)$, as the following simple example shows.

\begin{example}\label{e:1}
Suppose that $G$ is a nonabelian group of order $pq$, where $p$ and $q$ are primes, and $\pi=\{p,q\}$. Then $d_\pi(G/Z(G))=d(G)=2$ but $\epsilon_\pi(G)=1$.
\end{example}

Nevertheless, as our second main result shows, the $\pi$-length of $G/Z(G)$ can be bounded by $\epsilon_\pi(G)$ directly, even if $G$ is just a $\pi$-separable group.

\begin{theorem}\label{t:pilength}
If $G$ is a finite $\pi$-separable group, then $l_\pi(G/Z(G))\leq\epsilon_\pi(G)$.
\end{theorem}

It is worth noting that each of the main results of the paper becomes false if one tries to replace $G/Z(G)$ by $G$ in the left part of the corresponding inequality. To see this, it suffices to take an abelian group $G$ whose order has a nontrivial $\pi$-part.

\section{The Frattini length and maximum $p$-index}\label{s:fp}

In this section, we prove Theorem~\ref{t:frattini}. We begin with two properties of conjugacy classes that are readily seen.

\begin{lemma}\label{l:fixP}
Let $P$ be a Sylow $p$-subgroup of a group $G$, and $x\in G$. If $\epsilon_p(x)=m$, then there is an element $z\in x^G$ such that $|P:P\cap C_G(z)|=p^m$.
\end{lemma}

\begin{proof} Let us take a Sylow $p$-subgroup $Q$ of $G$ maximal with respect to the size of the intersection $Q\cap C_G(x)$. Since $|G|_p=|Q|$ and $|C_G(x)|_p=|Q\cap C_G(x)|$, it follows that $|Q:Q\cap C_G(x)|=|G:C_G(x)|_p=p^m$. By the Sylow theorem, there is $y\in G$ with $P=Q^y$. Therefore, for $z=x^y$,
$
|P:P\cap C_G(z)|=|Q^y:Q^y\cap C_G(x^y)|=|Q:Q\cap C_G(x)|=p^m.
$
\end{proof}

\begin{lemma}\label{l:reps}
If $\{z_1,\ldots,z_s\}$ is a complete set of representatives of
conjugacy classes of a group $G$, then $G=\langle z_1,\ldots,z_s\rangle$.
\end{lemma}

\begin{proof} It was noted by Burnside in~\cite[\S~26]{BurnsideBook}.
\end{proof}

Now we turn to the necessary properties of the Frattini series. Recall that for a $p$-group~$G$, the Frattini subgroup $\Phi(G)$ of $G$ can be also defined as $\Phi(G)=[G,G]G^p$, where $[G,G]$ is the derived subgroup of $G$ and $G^p=\langle x^p : x \in G\rangle$, see, e.g., \cite[Defintion~4.6]{98Kh} or \cite[Theorem~5.1.3]{68Gor}.

\begin{lemma}\label{l:de}
$d_p(G)\leq\phi_p(G)$ and $e_p(G)\leq\phi_p(G)$.
\end{lemma}

\begin{proof} Suppose that $P$ is a Sylow $p$-subgroup of a group $G$. If $P^{(m)}=[P^{(m-1)},P^{(m-1)}]$ is $m$th member of the derived series of $P$, where $P^{(0)}=P$, and $P^{p^m}=\langle x^{p^m}:x\in P\rangle$, then the later definition of the Frattini subgroup of a $p$-group implies inductively that $P^{(m)}P^{p^m}\leq\Phi_m(P)$, and we are done.
\end{proof}

\begin{lemma}\label{l:frattini}
If $H$ is a subgroup of a $p$-group $G$, then $\Phi(H)$ is a subgroup of $\Phi(G)$.
\end{lemma}

\begin{proof}
It is clear that $[H,H]\leq[G,G]$ and $H^p\leq G^p$, so $\Phi(H)=[H,H]H^p\leq[G,G]G^p\leq\Phi(G)$, as required.
\end{proof}

\begin{lemma}\label{l:indexandfrat}
If $H$ is a subgroup of a $p$-group $G$ with $|G:H|=p^m$, then $\Phi_m(G)\leq H$.
\end{lemma}

\begin{proof} Induction on $m$. The claim is clear when $m=0$ and $H=G$. So we may assume that $m>0$ and $H$ is a proper subgroup of $G$. Since $G$ is a $p$-group, there is a subgroup $K$ of $G$ such that $H$ is a subgroup of $K$ and $|K:H|=p$. Then $|G:K|=p^{m-1}$ and the inductive hypothesis implies that $\Phi_{m-1}(G)\leq K$. Since $H$ is a maximal subgroup of $K$, it follows that $\Phi(K)\leq H$. Now Lemma~\ref{l:frattini} yields $\Phi_m(G)=\Phi(\Phi_{m-1}(G))\leq\Phi(K)\leq H$, and we are done.
\end{proof}

\begin{lemma}\label{l:fratind}
If $H$ is a $p$-subgroup of a group $G$, and $Z=Z(G)$, then $\Phi(HZ/Z)=\Phi(H)Z/Z$.
\end{lemma}

\begin{proof} Since $[HZ,HZ]=[H,H]$ and $(HZ)^p=H^pZ^p$, it follows that $[HZ/Z,HZ/Z]=[H,H]Z/Z$ and $(HZ/Z)^p=H^pZ/Z$. Therefore,
$$
\Phi(HZ/Z)=[HZ/Z,HZ/Z](HZ/Z)^p=[H,H]H^pZ/Z=\Phi(H)Z/Z,
$$ as required.
\end{proof}

\begin{lemma}\label{l:fratindm}
Let $P$ be a $p$-subgroup of a group $G$, and $Z=Z(G)$. Then $\Phi_m(PZ/Z)=\Phi_m(P)Z/Z$.
\end{lemma}

\begin{proof} Induction on $m$. It is clear for $m=0$. Now $m>0$ and, by the inductive hypothesis, $\Phi_{m-1}(PZ/Z)=\Phi_{m-1}(P)Z/Z$. Applying Lemma~\ref{l:fratind} for $H=\Phi_{m-1}(P)$, we have
$$
\Phi_m(PZ/Z)=\Phi(\Phi_{m-1}(PZ/Z))=\Phi(\Phi_{m-1}(P)Z/Z)=\Phi(\Phi_{m-1}(P))Z/Z=\Phi_m(P)Z/Z,
$$
as required.
\end{proof}

We are in position to finish the proof.\medskip

{\em Proof of  Theorem~{\rm\ref{t:frattini}}}. Let $\epsilon_p(G)=m$. Then for every $x\in G$, $|x^G|_p=|G:C_G(x)|_p\leq p^m$.

Let us fix some Sylow $p$-subgroup $P$ of $G$. By Lemma~\ref{l:fixP}, there is $z\in x^G$ with
$$
|P : P\cap C_G(z)|\leq p^m.
$$
By Lemma~\ref{l:indexandfrat},
$$
\Phi_m(P)\leq P\cap C_G(z)\leq C_G(z).
$$
Since this is true for all $x\in G$, Lemma~\ref{l:reps} implies that one can take elements $z_1,\ldots,z_s$ of $G$ such that $G=\langle z_1,\ldots, z_s\rangle$ and $\Phi_m(P)\leq C_G(z_i)$ for every $i=1,\ldots,s$. It follows that $\Phi_m(P)\leq Z$, where $Z=Z(G)$. Lemma~\ref{l:fratindm} yields
$$
\Phi_m(PZ/Z)=\Phi_m(P)Z/Z=1,
$$
and we are done.

\section{The $\pi$-length and maximum $\pi$-index}\label{s:pi}

The goal of this section is to prove Theorem~\ref{t:pilength}.

\begin{lemma}\label{l:cenineq}
Let $N$ be a normal subgroup of a group $G$, and $\ov{G}=G/N$. If $\ov{x}$ is an image of an element $x\in G$ in $\ov{G}$, then $|\ov{x}^{\ov{G}}|$ divides $|x^G|$. Moreover, if $C_G(x)\cap N$ is a proper subgroup of~$N$, then $|\ov{x}^{\ov{G}}|$ is a proper divisor of $|x^G|$.
\end{lemma}

\begin{proof} It is clear that $C_{\ov{G}}(\ov{x})$ includes $C_G(x)N/N$. Therefore,
$$
|\ov{x}^{\ov{G}}|=|\ov{G}:C_{\ov{G}}(\ov{x})|\mbox{ divides }|G/N:C_G(x)N/N=|G:C_G(x)|/|N:C_G(x)\cap N|\mbox{ divides }|x^G|,
$$
and $|\ov{x}^{\ov{G}}|$ is a proper divisor of $|x^G|$ provided $C_G(x)\cap N$ is a proper subgroup of $N$.
\end{proof}

\begin{lemma}\label{l:indineq}
Let $N$ be a normal subgroup of a group $G$, and $\ov{G}=G/N$. Then $\epsilon_\pi(\ov{G})\leq\epsilon_\pi(G)$. If $N$ is also a proper $\pi$-subgroup of $G$ with $C_G(N)\leq N$, then $\epsilon_\pi(\ov{G})<\epsilon_\pi(G)$.
\end{lemma}

\begin{proof} The first statement follows from the first statement of Lemma~\ref{l:cenineq} directly. Suppose that $N$ is also a proper $\pi$-subgroup of $G$ with $C_G(N)\leq N$. Then for every $x\in G\setminus N$, the intersection $N\cap C_G(x)$ is a proper subgroup of $N$. Hence $$|\ov{x}^{\ov{G}}|_{\pi}\mbox{ divides }|G:C_G(x)|_{\pi}/|N:C_G(x)\cap N|_{\pi},$$ which is a proper divisor of $|x^G|_{\pi}.$
Since $N<G$, it follows that $\epsilon_\pi(\ov{G})<\epsilon_\pi(G)$, as claimed.
\end{proof}

The next two statements are the famous Hall--Higman Lemma 1.2.3 \cite{56HalHig} reformulated for $\pi$-separable groups, see, e.g., \cite[Theorem~6.3.2]{68Gor}, and the direct corollary of this lemma.

\begin{lemma}\label{l:HH}
If $G$ be a $\pi$-separable group  and $\ov{G}=G/O_{\pi'}(G)$, then $C_{\ov{G}}(O_\pi(\ov{G}))\leq O_\pi(\ov{G})$. In particular, if $O_{\pi'}(G)=1$, then $C_G(O_\pi(G))\leq O_\pi(G)$.
\end{lemma}

Let $O_{\pi',\pi}(G)$ be the preimage in $G$ of $O_\pi(G/O_{\pi'}(G))$.

\begin{lemma}\label{l:center}
If $G$ is a $\pi$-separable group, then $Z(G)\leq C_G(O_{\pi',\pi}(G))\leq O_{\pi',\pi}(G)$. Moreover, if $Z(G)=O_{\pi',\pi}(G)$, then $G=Z(G)$ is abelian.
\end{lemma}

\begin{proof} The first claim follows from Lemma~\ref{l:HH} and the fact that $$C_G(O_{\pi',\pi}(G))O_{\pi'}(G)/O_{\pi'}(G)\leq C_{G/O_{\pi'}(G)}(O_{\pi',\pi}(G)/O_{\pi'}(G)).$$ Suppose that $Z(G)=O_{\pi',\pi}(G)$. Then $G=C_G(Z(G))= C_G(O_{\pi',\pi}(G))\leq O_{\pi',\pi}(G)=Z(G)$, as required.
\end{proof}

We are ready to prove the theorem now.\medskip

{\em Proof of Theorem~{\rm\ref{t:pilength}}}. First, observe that the statement of the theorem is true for all sufficiently small groups $G$, even for all abelian groups, so we may proceed by induction on the order of $G$ and assume that $G$ is nonabelian.

Consider the upper $\pi$-series~\eqref{eq:r} of $G$. First, we suppose that $K_0=O_{\pi'}(G)=1$. Then $P_1=O_\pi(G)\neq1$. Lemmas~\ref{l:HH} and~\ref{l:center} imply that $C_G(P_1)\leq P_1$ and $Z(G)<P_1$. In particular, $\epsilon_\pi(G)\geq1$ and $l_\pi(G/Z(G))=l_\pi(G)$. If $G=P_1$, then $l_\pi(G/Z(G))=1$ and we are done. Thus, $G>P_1$.

Now, on the one hand, Lemma~\ref{l:indineq} yields $\epsilon_\pi(G)>\epsilon_\pi(G/P_1)$, i.e., $$\epsilon_\pi(G/P_1)+1\leq\epsilon_\pi(G).$$
On the other hand, $P_2/P_1=O_{\pi',\pi}(G/P_1)$, so Lemma~\ref{l:center} guarantees that $Z(G/P_1)<P_2/P_1$ or $G/P_1$ is abelian. If $Z(G/P_1)<P_2/P_1$, then $Z(G/P_1)\leq K_1/P_1$, so $l_\pi((G/P_1)/Z(G/P_1))=l_\pi(G/K_1)=l_\pi(G/P_1)=l_\pi(G)-1=l_\pi(G/Z(G))-1$. If $G/P_1$ is abelian, then it must be a $\pi'$-group, because $P_1\neq1$. Hence $l_{\pi}((G/P_1)/Z(G/P_1))=0=l_\pi(G/Z(G))-1$. Thus, in both cases, $$l_\pi(G/Z(G))\leq l_\pi((G/P_1)/Z(G/P_1))+1.$$

Applying the inductive hypothesis to $G/P_1$, we obtain
$$
l_\pi(G/Z(G))\leq l_\pi((G/P_1)/Z(G/P_1))+1\leq\epsilon_\pi(G/P_1)+1\leq\epsilon_\pi(G),
$$
as required.

Thus, to complete the proof of the theorem it suffices to consider the case, where $K_0\neq1$. Put $\ov{G}=G/K_0$ and note that $\ov{G}$ is nontrivial, because the theorem holds for every $\pi'$-group $G$. Observe also that $O_{\pi'}(\ov{G})=1$. We may suppose that $l_\pi(\ov{G}/Z(\ov{G}))<l_\pi(G/Z(G))$. Otherwise, by induction on the order of $G$, $l_\pi(G/Z(G))=l_\pi(\ov{G}/Z(\ov{G}))\leq\epsilon(\ov{G})\leq\epsilon(G)$, where the last inequality follows from Lemma~\ref{l:indineq}.

Put $\ov{P}_1=P_1/K_0=O_\pi(\ov{G})$. Lemma~\ref{l:center} yields $Z(\ov{G})\leq\ov{P}_1$. If $Z(\ov{G})<\ov{P}_1$, then $l_\pi(\ov{G}/Z(\ov{G}))=l_\pi(G/Z(G))$, a contradiction.
Hence $Z(\ov{G})=\ov{P}_1$. Again by Lemma~\ref{l:center}, it is possible only if $Z(\ov{G})=\ov{P}_1=\ov{G}$. In particular, $G=P_1$. Since $l_\pi(G/Z(G))\leq1$ in this case, it suffices to show that $\epsilon_\pi(x)>0$ for some element $x\in G$. Otherwise, $|G:C_G(x)|_{p}=1$ for every $x\in G$ and $p\in\pi$. Then, by virtue of \cite[Lemma~1]{72Camina}, Sylow $p$-subgroups of $G$ lie in $Z(G)$ for all $p\in\pi$. It follows that $G/Z(G)$ is a $\pi'$-group, so $l_\pi(G/Z(G))=0$, which completes the proof.
\medskip

{\bf Acknowledgments.} A-Ming Liu was supported by the Hainan Provincial Natural Science Foundation (Grant no.~423RC429) and the National Natural Science Foundation of China (No.~12101165), A. V. Vasil'ev was supported by the Sobolev Institute of Mathematics state contract (project FWNF-2022-0002). The work was also supported by the National Natural Science Foundation of China (No.~12171126).\smallskip

The authors thank A. A. Buturlakin for helpful suggestions, improving the text of the paper.

\end{document}